\documentclass[11pt]{amsart}
\usepackage{verbatim}
\usepackage{amsmath}
\newtheorem{theorem}{Theorem}[section]

\newtheorem{prop}{Proposition}[section]

\newtheorem{lemma}{Lemma}[section]

\theoremstyle{definition}
{}

\theoremstyle{remark}

\newcommand{\C}{{\mathbb{C}}}

\newcommand{\ddc}{dd^c}
\DeclareMathOperator{\Tr}{Tr}
\numberwithin{equation}{section}
\usepackage{hyperref}

\usepackage[utf8]{inputenc}
\usepackage{mathtools,amssymb,amsthm}
\usepackage{mathrsfs}
\usepackage{microtype}
\usepackage{enumitem}
\newcommand{\al}{\alpha}

\newcommand{\ga}{\gamma}
\newcommand{\la}{\lambda}

\newcommand{\ep}{\epsilon}

\newcommand{\RR}{\mathbb{R}}
\newcommand{\CC}{\mathbb{C}}

\newcommand{\diag}{\mathrm{diag}}
\newcommand{\uv}{\underline{v}}
\newcommand{\ov}{\overline{v}}

\newcommand{\sB}{\mathcal{B}}
\newcommand{\Rm}{\mathrm{Rm}}

\DeclareMathOperator{\tr}{tr}
\DeclareMathOperator{\Ric}{Ric}

\DeclareMathOperator{\Herm}{Herm}
\DeclareMathOperator{\supp}{supp}
\DeclareMathOperator{\Lip}{Lip}

\hypersetup{colorlinks,citecolor=blue,plainpages=false,hypertexnames=false}
\begin{document}
\title{Positive bisectional curvature, Ricci mass, and the fundamental group}
\author[V. Datar]{Ved Datar}
\author[V. P. Pingali]{Vamsi Pritham Pingali}
\author[H. Seshadri]{Harish Seshadri}

\address{Department of Mathematics, Indian Institute of Science, Bengaluru, India}
\email{vvdatar@iisc.ac.in}
\email{harish@iisc.ac.in}
\email{vamsipingali@iisc.ac.in}
\begin{abstract} 
We prove the finiteness of the integral of the top power of the Ricci form on a complete, noncompact K\"ahler manifold with positive bisectional curvature. This settles a special case of a conjecture of Yau and Yang. An immediate corollary is the finiteness of the fundamental group of such manifolds.  As in our previous work, the main idea involves the construction and use of plurisubharmonic weight functions with finite Monge-Amp\`ere masses. 
\end{abstract}
\maketitle
\section{Introduction}\label{sec:intro}
A question of Yang \cite{Yang}, motivated by a conjecture of Yau \cite{Yau} for Riemannian manifolds, asks if 
$$\int_X Ric_{\omega}^n <\infty$$ for any complete noncompact K\"ahler $n$-manifold $(X, \omega)$ with nonnegative Ricci curvature. 

In a previous work \cite{DPS3}, we had answered this question affirmatively for K\"ahler manifolds with 
positive sectional curvature. In this paper, we extend the finiteness result to manifolds with positive holomorphic bisectional curvature, which we denote by $BK>0$.

\begin{theorem}\label{thm:finitericciintegral}
Let $(X,\omega)$ be a noncompact complete K\"ahler $n$-manifold with positive bisectional curvature. Then $$ \int_X Ric_{\omega}^n <\infty.$$
\end{theorem}

As a corollary, we have:

\begin{theorem}\label{thm:Finitepi1}

Let $(X,\omega)$ be a noncompact complete K\"ahler manifold with positive holomorphic bisectional curvature. Then the fundamental group of $X$ is finite.
\end{theorem}

Indeed, if $\pi_1(X)$ is infinite, then passing to the universal cover and applying Theorem \ref{thm:finitericciintegral} we arrive at a contradiction. In particular, $\mathbb{C}\times \mathbb{C}^*$ does not admit a complete K\"ahler metric of positive bisectional curvature.  We recall that a conjecture of Yau \cite{Y} states that complete noncompact K\"ahler $n$-manifolds with $BK>0$ are biholomorphic to ${\mathbb C}^n$. However, it is not even known if such manifolds are simply-connected.

We now make some remarks on the proof of Theorem \ref{thm:finitericciintegral}: An earlier work of the authors \cite{DPS3} (based on work by Chen and Zhu \cite{CZ2}) establishes the following proposition implying finiteness of Ricci mass. 

\begin{prop}\label{thm:phi}
Let $(X,\omega)$ be a complete, noncompact K\"ahler $n$-manifold with $BK>0$. If there exists a Lipschitz plurisubharmonic function $\phi$ on $X$ satisfying $$\Lip(\phi) \leq 1 \text{ and } 0< \int_X(dd^c\phi)^n < \infty,$$ in the sense of Bedford-Taylor, then $$\displaystyle\int_X Ric_{\omega}^n <\infty.$$
\end{prop}

In view of Proposition \ref{thm:phi} it is enough to produce a uniformly Lipschitz psh function with finite Monge-Amp\`ere mass:

\begin{theorem}\label{thm:MA}
Let $(X,\omega)$ be a complete, noncompact K\"ahler $n$-manifold with $BK>0$. There exists a Lipschitz plurisubharmonic function $\phi$ on $X$ satisfying $$\Lip(\phi) \leq 1 \text{ and } 0< \int_X(dd^c\phi)^n < \infty,$$ in the sense of Bedford-Taylor.
\end{theorem}

In \cite{DPS1}, positive  sectional curvature was used to obtain a smooth strictly plurisubharmonic (psh) exhaustion. This was employed to produce a Lipschitz function with finite Monge-Amp\`ere mass.  The exhaustion function was needed to provide barrier functions to solve the Monge-Amp\`ere equation and for the gradient estimate on the boundary. The key new insight (due to ChatGPT Astra Pro) is to ``approximate" the Monge-Amp\`ere equation by \emph{uniformly elliptic equations} using (truncated) Bellman operators. Nonlinear uniformly elliptic equations can be solved without undue boundary/barrier requirements. 

\subsection*{Ongoing work:} We can also prove Yau's uniformisation conjecture in full in complex dimension two, with substantial assistance from AI, and are currently writing up the proof. We hope to announce this result soon. OpenAI has recently announced a claimed proof of the conjecture in all dimensions, but its approach is quite different from ours. Our proof further develops the methods introduced in our previous work and in the present paper, and we believe these ideas are of independent interest and merit a separate exposition. We also note that the main theorem of the present paper, i.e., Theorem  \ref{thm:finitericciintegral} remains new even in light of the recent OpenAI announcement. 

\subsection*{Acknowledgements}: The second-named author Pingali is grateful to ANRF for support via a MATRICS grant (ANRF/ARGM/2025/000313/MTR). The authors also acknowledge partial support by the DST FIST program-2021
[TPN-700661].\\

\subsection*{AI usage:} ChatGPT Astra Pro was used extensively, over several conversations, in exploring possible approaches to the proof of the main theorem. Its main useful contribution was to suggest the use of truncated Bellman operators as a uniformly elliptic approximation to the degenerate complex Monge–Ampère operator. The arguments it proposed for carrying out this program, however, relied heavily on viscosity solutions and were informal at several important points. Rather than trying to verify or repair those arguments, we retained the Bellman-approximation idea and developed the PDE argument independently. Our proof uses standard regularity theory for concave uniformly elliptic equations, including Caffarelli-type interior estimates, together with a Błocki-type gradient estimate. The latter uses positive bisectional curvature in essentially the same way as in our earlier work [5], and quite differently from the AI-generated argument.

\section{A Lipschitz function with finite Monge-Amp\`ere mass}\label{sec:MongeAmpere}

The goal of this section is to prove Proposition \ref{thm:phi}. For the remainder of this paper, for every  smooth real function $w$, let $H(w)$ denote the Hermitian endomorphism associated with $\sqrt{-1}\partial\bar\partial w$ relative to $\omega$. Writing $dd^c=\sqrt{-1}\partial\bar\partial$, we have
\begin{equation}\label{eq:ma-limit-normalization}
 (dd^c w)^n=\det H(w)\,\omega^n.
\end{equation}

\subsection{Truncated Bellman operators} Let $f\in C^\infty(M)$. Our goal is to find a uniformly Lipschitz function $\phi$ satisfying: 
\begin{equation}\label{eq:main-cma}
(dd^c\phi)^n = e^{-f}\omega^n,~ dd^c\phi\geq 0,
\end{equation} where the equation is interpreted in the Bedford-Taylor \cite{BT} sense. We solve the equation by using the Bellman method, i.e., approximating the complex Monge-Ampere equation by simpler concave elliptic PDEs. The starting point is the elementary linear algebra observation:  If $A \in M(n, \C)$ is  positive-semidefinite Hermitian, then 

\begin{equation}\label{bell}
(\det A)^{1/n}=  \frac{1}{n}\inf\{\tr(BA)~|~ B>0,~ \det B = 1\}.
\end{equation}

We let 
\begin{align*}
\sB&:= \{B\in \Herm^+(n,\CC)~|~ \det B = 1\}\\
F(A) &:= \frac{1}{n}\inf_{B\in \sB}\Tr(AB),\qquad A \geq 0.
\end{align*}

A smooth function $\phi$ is then a solution to \eqref{eq:main-cma} if and only if $A = H(u)$ satisfies 
\begin{equation}\label{eq:main-cma-F}
F(A) = h
\end{equation} where $h = e^{-f/n}$. As we shall see below, any smooth solution to the above equation is automatically psh. Of course, obtaining a smooth solution in this generality may be too optimistic. There are two additional difficulties in solving the above equation: $F$ is not smooth and is not uniformly elliptic. The latter condition is important for solving the equation on general domains, as opposed to pseudoconvex domains. So the work-around is to approximate $F$ by uniformly elliptic smooth concave operators $F_\ep.$ We need the following result.
\begin{lemma}
There exists a family of smooth maps $F_\ep:\Herm(n,\CC)\rightarrow \RR$ satisfying the following properties: 
\begin{enumerate}
\item $F_\ep$ are concave, i.e., for all $A,A_0\in \Herm(n,\CC)$ we have $$F_\ep(A_0) \leq F_\ep(A) + DF_\ep(A)(A_0-A).$$
\item (Uniform ellipticity) For every $P\geq 0$, $$\frac{\ep}{n}\Tr(P) \leq F_\ep(A+P) - F_\ep(A) \leq \frac{1}{n\ep}\Tr(P).$$
\item (Uniform approximation) There exists a constant $C = C(n)$ such that for all compact sets $K\subset \Herm^+(n,\CC)$ there exists a constant $\ep_K$ with the following property: for all $\ep < \ep_K$ we have $$\Vert F_\ep - F\Vert_{L^\infty(K)}< C\ep^2.$$
\end{enumerate}
\end{lemma}

\begin{proof}
For $0 < \ep \le 1$, we consider the truncated Bellman family: \begin{align*}
\sB_\ep&:= \{B\in \Herm^+(n,\CC)~|~ \det B = 1,~\ep I\leq B\leq \ep^{-1}I\}\\
\widehat{F}_\ep(A) &:= \frac{1}{n}\min_{B\in \sB_\ep}\Tr(AB).
\end{align*}
Then $\widehat{F}_\ep$ is concave and uniformly elliptic with the constants $\ep/n$ and $1/n\ep$. To obtain $F_\ep$ we use the standard method of mollifying with an approximation to identity. Since the uniform estimate is crucial to the final step of the proof, we include the details for the convenience of the reader. Identifying $\Herm(n,\CC)$ with $\RR^{n^2}$ we let $\eta$ be the standard radial mollifier satisfying $$\eta \in C^\infty_c(\Herm(n,\CC)), \quad \eta \ge 0, \quad \mathrm{Supp} \ \eta \subset \{\vert Z\vert<1\},\quad \int_{\Herm(n,\CC)}\eta(Z)\,dZ = 1,$$ and $$\eta_\sigma(Z) = \sigma^{-n^2}\eta\Big(\frac{Z}{\sigma}\Big).$$ We then let $$F_{\ep,\sigma}(A) = \widehat F_\ep \star\eta_\sigma(A):=\int_{\Herm(n,\CC)}\eta(Z)\widehat F_\ep(A-\sigma Z)\,dZ.$$ It is easy to see that $F_{\ep,\sigma}$ is concave and uniformly elliptic with the same constants $\ep/n$ and $1/n\ep$. The ellipticity of $\widehat F_\ep$ also implies that $\Lip(\widehat F_\ep) \leq C_n/\ep$ for some dimensional constant $C_n$ and it then follows immediately that for any $A\in \Herm(n,\CC)$, we have $$|F_{\ep, \sigma}(A) - \widehat F_\ep(A)| \leq \frac{C_n\sigma}{\ep}.$$ We now pick $\sigma = \ep^3$ and set $F_\ep = F_{\ep,\sigma}.$ Then 
\begin{equation}\label{eq:bellman-smoothening-uniform-est} 
|F_\ep(A) - \widehat F_\ep(A)| \leq C_n \ep^2\quad\text{ for all $A\in \Herm(n,\CC)$.}
\end{equation}
Finally observe that given a compact set $K\subset \Herm^+(n,\CC)$, then there exists an $\ep_K>0$ such that for all $\ep< \ep_K$ we have $\widehat F_\ep(A) = F(A).$ To see this, note that for any $A\in \Herm^+(n,\CC)$ the Bellman minimiser is completely explicit. Indeed, one can take $$B_A = (\det A)^{1/n}A^{-1}.$$ So if $K\subset \Herm^+(n,\CC)$ is compact, then there exist $\ep_K$ such that $$\ep_K \leq B_A \leq \ep_K^{-1}\quad \text{ for all $A\in K$},$$ and so $\widehat F_\ep(A) = F(A)$ for all $A\in K$ and $\ep< \ep_K$.
\end{proof}
We'll need the following two basic properties of the operators $F_\ep$. 
\begin{lemma}
\begin{enumerate}
\item For every $A\in \Herm(n,\CC)$, $\Tr(DF_\ep(A)) \geq 1.$
\item For $n\geq 2$, if $A_\ep\rightarrow A$ as $\ep\rightarrow 0^+$, and $\la_{min}(A) <0 $, then $$F_\ep(A_\ep) \rightarrow -\infty.$$
\end{enumerate}
\end{lemma}
\begin{proof}
The first part follows from the following finite difference inequality and the smoothness of $F_\ep$: 
\begin{equation}\label{eq:finite-diff-F-ep}
F_\ep(A + tI) - F_\ep(A) \geq t,
\end{equation} for $t \ge 0$ and the reversed inequality for $t \le 0$. It is enough to prove this finite difference inequality with $F_\ep$ replaced by the truncated Bellman operators $\widehat F_\ep$ used above. For $\widehat F_\ep$ we have 
\begin{align*}
\widehat F_\ep(A+tI) &= \frac{1}{n}\inf_{B\in \sB_\ep} \Tr(AB + tB)\\
&= \frac{1}{n}\inf_{B\in \sB_\ep} (\Tr(AB) + t\Tr(B)).
\end{align*}
By the AM-GM inequality, since $\det B = 1$ we have $$\tr(B) \geq n,$$ and so $$\widehat F_\ep(A+tI) \geq \widehat F_\ep(A) + t.$$

For the second part, we claim that  if $A_\ep\rightarrow A$ as $\ep\rightarrow 0^+$, and $\la_{min}(A) <0 $, then 
\begin{equation}\label{eq:bellman-neg-ev}
\widehat F_\ep(A_\ep) \rightarrow -\infty.
\end{equation} The required result then follows from combining this with the estimate \eqref{eq:bellman-smoothening-uniform-est}  above. To prove \eqref{eq:bellman-neg-ev} suppose there is a unit eigenvector $e_1$ such that $$\langle Ae_1,e_1\rangle = -2a<0.$$ Then for $\ep<<1$, $\langle A_\ep e_1,e_1\rangle < -a.$ Also $|A_\ep| \leq M$ for some $M>0$. Assume $n \ge 2$ and extend $e_1$ to an orthonormal basis $\{e_1,\cdots,e_n\}$ and let $$B_\ep = \diag(\ep^{-1},\ep^{1/(n-1)},\cdots, \ep^{1/(n-1)}).$$ Then for small $\ep$, $B_\ep \in \sB_\ep$ and so we have 
\begin{align*}
\widehat F_\ep(A_\ep) &\leq \frac{1}{n}\Tr(A_\ep B_\ep)\\
&\leq \frac{1}{n}\Big(-\frac{a}{\ep} + \ep^{1/(n-1)}\sum_{i=2}^nA_{i\bar i}\Big)\\
&\leq -\frac{a}{n\ep} + \sqrt{n}M\ep^{1/(n-1)}.
\end{align*}
Taking $\ep\rightarrow 0$ we prove \eqref{eq:bellman-neg-ev}.
\end{proof}

\subsection{Preliminary constructions} To solve the relevant PDEs and for estimates we need two constructions which are essentially contained in our previous work, but are reproduced here for convenience. Throughout, we let $B:X\rightarrow (0,\infty)$ the function defined by $$B(x):= \inf_{u,v\in S_xM}\frac{\Rm(u,v,v,u) + \Rm(u,Jv,Jv,u)}{1+ \langle u,v\rangle^2 + \langle u,Jv\rangle^2},$$ where $S_xX$ is the unit sphere in the real tangent space $T_xX$. Note that the convention implies that $\Ric \geq (n+1)B.$ 

Note that our condition $BK>0$ means $B(x)>0$ for all $x \in X$.

\begin{lemma}\label{lem:cutoffs}
Let $(X,\omega)$ be a complete noncompact K\"ahler manifold with $BK>0$. 
\begin{enumerate}
\item There exists a smooth exhaustion function $\rho$ and constant $C_\rho>0$ such that  $$\vert\nabla\rho\vert+|H(\rho)|\le 1.$$
\item There exists a $R_0>0$ such that for any $R>R_0$ there is a function $\chi_R\in C^\infty$ satisfying \begin{equation*}\label{eq:cutoffs}
 \chi_R=1\text{ on }B(o,R),\ \ \ 
 0 \le  \chi_R \le 1, \ \ \ 
 \supp\chi_R\subset B(o,2R),\ \  \
 |\nabla\chi_R|+|\ddc\chi_R|\le C/R.
\end{equation*}
In particular, $\chi_R\to1$ locally and for every fixed integer $b\ge2$, the exists $C_b >0$ such that 
\begin{equation*}\label{eq:cutoff-powers}
 -\frac{C_b}{R}\chi_R^{b-2}\omega
 \le\ddc(\chi_R^b)
 \le\frac{C_b}{R}\chi_R^{b-2}\omega.
\end{equation*}
Here $B$ is the $\rho$-ball.
\item There is a smooth function $u_0: X  \rightarrow {\mathbb R}$ and a constant $C_0$ (depending on $u_0$)   such that
$$\vert \nabla u_0\vert\leq 1 \text{ and }0 < H(u_0) < C_0I.$$
\end{enumerate}
\end{lemma}
\begin{proof}
(1) is just Proposition 3.1 of \cite{NT3}: time 1 heat-flow of a distance function gives a smooth exhaustion function $\rho$ with $$|\nabla \rho|, | H(\rho)| \leq C .$$ For the second part, one simply takes $\chi = \eta(\rho/R)$ where $\eta$ is a standard one-variable cut-off function. For the third part, instead of the distance function, we take the heat flow at unit time of a Buseman function given by $$b(x) = \lim_{t\rightarrow \infty}(t - d(x,\ga(t)),$$ where $\ga$ is a ray. It is well known \cite{Wu, GW, NT3} that if $BK>0$, $H(P_1b) >0$. Then $u_0 = P_1b$, where $P_1b$ is the time 1 heat-flow of $b$, satisfies the required estimates once again by the aforementioned work of Ni and Tam \cite{NT, NT2}. 
\end{proof}

For the estimates, we will also need a barrier function that is better behaved than $u_0$. This is furnished by the following result. 

\begin{lemma}\label{lem:betterbarrier}
For each $a>0$ there exists a smooth function $\zeta_a$ solving $$\Delta\zeta_a = a(\zeta_a - u_0).$$ Moreover, there exists a dimensional constant $C_n$ such that $$|\zeta_a - u_0| \leq \frac{C_n}{\sqrt{a}}.$$
\end{lemma}
\begin{proof}
Let $P_t$ be the heat semi-group as above. We then let $$\zeta_a(x) = a\int_0^\infty e^{-at}P_tu_0(x)\,dt.$$ It is easy to check that this function solves the equation. Next, we note that since $\Ric \geq 0$ and $|\nabla u_0|\leq 1$, we have the standard estimate (cf. \cite[Lemma 6]{DPS1}) $$\Vert P_tu_0 - u_0\Vert_{L^\infty(X)} \leq C\sqrt{t}. $$ We then estimate 
\begin{align*}
|\zeta_a(x) - u_0(x)|  &\leq a\int_0^\infty e^{-at}|P_tu_0(x) - u_0(x)|\,dt\\
&\leq Ca\int_0^\infty e^{-at}\sqrt{t}\,dt\\
&= \frac{C'}{\sqrt{a}},
\end{align*}
for some $C'>0$.
\end{proof}

\begin{lemma}
Let $(X,\omega)$ be a complete K\"ahler manifold with $BK>0$. Then there exists an exhaustion function $f\in C^\infty(X,\RR)$ satisfying the following properties:
\begin{enumerate}
\item $\sup_{X}e^{-f/n} < \infty.$
\item $e^{-f/n}\vert \nabla f\vert  \leq \min(nB/4,1).$
\item $\int_X e^{-f}\omega^n < \infty.$
\item $0<h(x)<\frac{1}{2}m(x)$ where $m(x)=(\det(H(u_0))(x))^{1/n}$
\end{enumerate}
\end{lemma}
\begin{proof}
We reproduce the proof from \cite{DPS1}. Without loss of generality we may assume that $\inf_X \rho = 0$. Fix a point $o\in \rho^{-1}(0)$. We let $B_r = \{\rho < r\}$ and $S_r = \{\rho = r\}$. Let $g:[0,\infty)\rightarrow (0,\infty)$ be a smooth strictly decreasing function such that $$g(r) \leq \inf_{x\in S_r}e^{-d(o,x)},$$ and $$g(r)<\frac{1}{2} \inf_{\rho\leq r} m.$$ Note that since $\rho$ is an exhaustion, the infimum on the right is positive for each $r>0$. Note also that $\lim_{r\rightarrow \infty}g(r) = 0.$ We now set $$\xi(\rho) = -n\ln\Big(\int_{\rho}^\infty G(r)\,dr\Big), $$ where we choose $G$ to be a smooth function satisfying $$0< G(r) < \min\Big(-g'(r), \frac{1}{4}\inf_{x\in S_r}B(x), \frac{1}{n}\Big).$$ 
Then $f(x) = \xi(\rho(x))$ is the required function.

\end{proof}

\subsection{Proof of the Theorem \ref{thm:MA}}

Let $c_j\rightarrow\infty$ such that the domains $\Omega_j = \{\rho < c_j\}$ are smooth. For $\delta>0$ consider the following Dirichlet problem: $$\begin{cases} F_\ep(H(u)) - \delta(u - u_0) = h\\ u = u_0 \text{ on }\partial\Omega_j,\end{cases}$$ where $h = e^{-f/n}>0$ and $f$ is as in the Lemma above. By standard theory \cite{T, CKN} there exists a unique smooth solution $u = u_{\ep,\delta,j}$. Moreover $$u_{\ep,\delta,j}\xrightarrow{C^\infty_{loc}(X)} u_{\ep,\delta},$$ where $u_{\ep,\delta}$ solves $$F_\ep(u_{\ep,\delta}) - \delta(u_{\ep,\delta} - u_0) = h. $$ For estimates we will further assume that 
\begin{itemize}
\item $\sup_X\vert h \vert < \infty$.
\item $|\nabla h| \leq B/2$.
\item $\int_X h\omega^n <\infty.$
\end{itemize} 
Such a function was constructed in \cite{DPS1} and we reproduce it at the end of the section. 
\begin{lemma}\label{lem:comparison}
There exists a constant $K = K(\Vert h\Vert_{L^\infty(X)},\ep)$ (but independent of $\delta$) such that $$\sup_X\vert u_{\ep,\delta} - u_0\vert < \frac{K}{\delta}.$$ 
\end{lemma}
\begin{proof}
It suffices to prove that $$\sup_X\vert u_{\ep,\delta,j} - u_0\vert < \frac{K}{\delta}$$ on $\overline{\Omega_j}$ as long as $K$ is independent of $j$. Let $p$ be a minimum of $v = u_{\ep,\delta,j} - u_0$. If $p$ is on the boundary, then $v(p) = 0$. Else $p$ is in the interior, and so by the maximum principle $H(u_{\ep,\delta,j}) \geq H(u_0)$  at $p$. Hence by ellipticity of $F$, $$h + \delta v(p) = F_\ep(H(u_{\ep,\delta,j}))(p) \geq  F_\ep(H(u_0))(p) \geq -C,$$ since $H(u_0)$ lies in a compact set and $F_\ep$ is smooth. In any case, we then have that $$\delta \inf v \geq -C - \Vert h\Vert_{L^\infty(X)}.$$ The upper bound is established in the same way. 
\end{proof}
\begin{prop}
$\sup_{X}\vert\nabla u_{\ep,\delta}\vert \leq 1.$
\end{prop}
\begin{proof}
Set $w = \vert\nabla u_{\ep,\delta}\vert^2$, $q = \delta(u-u_0) + h$, and let $L$ denote the linearised operator $$L = F_\ep^{i\bar j}\nabla_i\nabla_{\bar j}\text{ where }F_\ep^{i\bar j} = \frac{\partial F_\ep}{\partial_{\la_{i\bar j}}}(H(u)).$$ We first claim that 
\begin{equation}\label{eq:claim-bounded-grad}
\sup_{X}w < \infty.
\end{equation}Computing at a point in coordinates which are normal for $\omega$ and in which the Hessian of $u$ is diagonal (with say eigenvalues $\mu_1,\cdots,\mu_n$), we have 
\begin{align*}
Lw &= F_\ep^{i\bar j}(g^{a\bar b}u_au_{\bar b})_{i\bar j}\\
&= F_\ep^{i\bar j}\Big(u_{ai\bar j}u_{\bar a} + u_{a}u_{i\bar j \bar a} + u_{ia}u_{\bar j\bar a}+u_{a\bar j}u_{i\bar a} + R_{i\bar j a\bar b}u_a u_{\bar b}\Big)\\
&\geq F_\ep^{i\bar j}\Big(u_{ai\bar j}u_{\bar a} + u_{a}u_{i\bar j \bar a} + u_{ia}u_{\bar j\bar a}+u_{a\bar j}u_{i\bar a}\Big) + B |\nabla u|^2,
\end{align*}
where as usual we sum over repeated indices and we used the observation that $\sum_i F^{i\bar i} \geq 1.$ Crucially, the positive holomorphic bisectional curvature was used to get the last inequality.  Next, differentiating the equation we have that $F_\ep^{i\bar j}u_{i\bar j a} = q_a,$ and so setting $\la_\ep = \ep/n$ we have 
\begin{align*}
Lw &\geq q_au_{\bar a}+ q_{\bar a}u_a + F_\ep^{i\bar j}\Big(u_{ia}u_{\bar j\bar a}+u_{a\bar j}u_{i\bar a}\Big) + B|\nabla u|^2\\
&\geq q_au_{\bar a}+ q_{\bar a}u_a + \la_\ep\sum_{i,a}|u_{ia}|^2 +\la_\ep\sum_{k}\mu_k^2 + Bw\\
&=  \la_\ep\sum_{i,a}|u_{ia}|^2 +\la_\ep\sum_{k}\mu_k^2 + (B + 2\delta)w + 2\Re(h_a u_{\bar a}) - 2\delta\Re((u_0)_au_{\bar a})
\end{align*} Naively, if $w$ were to take a maximum value at some point, the above inequality would immediately imply an upper bound on $w$. Since in principle the maximum of $w$ could drift off to infinity, we instead consider $U = \eta_Rw$, where $\eta_R = \chi_R^2$ and $\chi_R$ is the cut-off function from Lemma \ref{lem:cutoffs}. Note that $$\frac{|\nabla \eta|^2}{\eta} \leq \frac{C}{R^2} \text{ and } |dd^c\eta| \leq \frac{C}{R}.$$Then (suppressing the subscript $R$)
\begin{align*}
LU &\geq \eta\la_\ep\sum_{i,a}|u_{ia}|^2 +\eta\la_\ep\sum_{k}\mu_k^2 + (B + 2\delta)U + 2\eta\Re(h_a u_{\bar a}) - 2\eta\delta\Re((u_0)_au_{\bar a})\\
&+ 2F^{i\bar j}\Re(\eta_{i}w_{\bar j}) + F_\ep^{i\bar j}w\eta_{i\bar j}\\
&\geq \eta\la_\ep\sum_{i,a}|u_{ia}|^2 +\eta\la_\ep\sum_{k}\mu_k^2 + (B + 2\delta)U  - 2C \eta\sqrt{w} +2F_\ep^{i\bar j}\Re(\eta_{i}w_{\bar j})  - \frac{C}{R\la_\ep}w.
\end{align*}
Denoting by $Q = \sum_{i,a}|u_{ia}|^2 + \sum_{k}\mu_k^2,$ we control the second-last term by 
\begin{align*}
|2F_\ep^{i\bar j}\Re(\eta_{i}w_{\bar j})| \leq C\la_\ep^{-1}|\nabla \eta| \sqrt{w}\sqrt{Q} \leq \frac{\eta\la_\ep}{2}Q + \frac{C|\nabla\eta|^2}{2\la_\ep^3\eta}w \leq \frac{\eta\la_\ep}{2}Q + \frac{C}{2\la_\ep^3 R^2}w,
\end{align*} and so $$LU \geq  - 2C \eta\sqrt{w}  +(B + 2\delta)U - \frac{C}{R}w,$$ where in the last term in the last line, the constant is of course allowed to depend on $\ep$ but not $R$. At this stage, if the final negative term had a $\eta w$ instead of just a $w$ we could have absorbed it in the second term by choosing $R$ large and an application of the maximum principle would have shown that $w$ is bounded. Instead, we use $B>0$ to simply drop the second term and write $$LU \geq - 2C \eta\sqrt{w} - \frac{C}{R}w.$$ Now we need some more positivity to absorb the terms with $-w$. So consider the function $$V = U + \psi(v),$$ where we set $v = u - u_0$ and $\psi$ will chosen shortly and satisfies $$\psi ' < 0 \text{ and }\psi'' \geq c>0$$ on the range of $v$. We now compute:  
\begin{align*}
LV &= LU + L\psi(v)\\
 &= LU + \psi'(v) Lv + \psi''(v)F_\ep^{i\bar j}v_iv_{\bar j}\\
 &\geq LU +  \psi'(v) Lv + c\ep |\nabla v|^2.
\end{align*} 
The key point now is to use concavity of $F_\ep$. Indeed, we have $$F_\ep(H(u_0)) \leq F_\ep(H(u)) + F_\ep^{i\bar j} ((u_0)_{i\bar j} - u_{i\bar j}),$$ and so $$Lv \leq F_\ep(H(u)) - F_\ep(H(u_0)) = h + \delta v - F_\ep(H(u_0)):= G.$$ Since $\psi'<0$ we then have that $$\psi'(v)Lv \geq \psi'(v)G.$$For the third term in the lower bound on $LV$, by AM-GM inequality we have $$|\nabla v|^2 \geq \frac{1}{2}w -C $$ using the fact that $|\nabla u_0|$ is uniformly bounded. Combined with the earlier estimate on $LU$ we have $$LV \geq - 2C \eta\sqrt{w} - \frac{C}{R}w +\psi'(v)G +  \frac{c\ep}{2}w - C.$$ Now we choose $R>>1$ such that $$LV \geq - 2C \eta\sqrt{w}  + \frac{c\ep}{4}w + \psi'(v)G - C.$$ At this point, we now pick $$\psi(x) = (M-x)^2,$$ where $M = \sup v + 1$. Then $\psi$ satisfies both the required properties (with $c  =2$). Note that $\psi'(x) = -2(M-x)$ and so $\psi'(v)<0$ and $\psi''(v) \geq 2$. Also $G$ has a uniform upper bound, and so in the end we have $$LV \geq - 2C \eta\sqrt{w}  + \frac{c\ep}{4}w - C.$$

We now consider the maximum of $V$ on $B(0,2R)$. If the maximum occurs on the boundary, $\eta_R$ is zero there and so, $$V =  \psi(v) \leq \sup_{X}\psi(v) \leq C_{\ep,\delta}.$$ Crucially the constant does not depend on $R$. On the other hand, if the maximum is attained on the interior of $B(o,2R)$, say at a point $p$, then by the maximum principle $$0 \geq - 2C \eta\sqrt{w}  + \frac{c\ep}{4}w - C,$$ and so $w(p) \leq C$ for some constant $C>0$ independent of $R$. But then in particular this implies that $$\sup_{B(o,R)}w \leq C$$ for some constant independent of $R$. This proves the claim \eqref{eq:claim-bounded-grad}. 

We now prove that in fact $w\leq 1$. In fact we claim that there exists a constant $C_\ep>0$ such that for any $\al>0$ we have 
\begin{equation}\label{eq:w-al}
w(x) \leq 1 + \al\Big(\rho(x) + \frac{C_\ep}{\delta}\Big).
\end{equation} We then let $\al\rightarrow 0$. To prove the claim, we consider the function $W = w - \al\rho$, where note that we have normalized so that $\rho\geq 0$. Since $\rho$ is an exhaustion, this certainly takes the maximum at some point, say at $p$. Then at $p$, $$Lw(p) \leq \al L\rho(p) \leq \al C_\ep$$If $w(p) \geq 1$, by our calculations earlier, and using the fact that $|\nabla h| \leq B?2$ and $|\nabla u_0| \leq 1$ we have 
\begin{align*}
Lw&\geq (B+2\delta)(w - \sqrt{w})(p)\\
&\geq \delta(w(p)-1),
\end{align*}
and so $$w(p) \leq 1 + \al\frac{C_\ep}{\delta},$$ and so this proves \eqref{eq:w-al}. If $w(p) \leq 1$ the same upper bound is automatic. 
\end{proof}

We now let $\varphi_{\ep,\delta} = u_{\ep,\delta} - u_{\ep,\delta}(o).$ Then by the Arzela-Ascoli theorem, as $\delta\rightarrow 0^+$ via a sequence, after passing to a diagonal subsequence, $\varphi_{\ep,\delta}$ converge locally uniformly to Lipschitz function $\varphi_\ep$ with $\Lip(\varphi_\ep)\leq 1.$ Next, we have the following basic interior estimate (cf. \cite[Theorem 3.1]{T}, \cite{Caf}, \cite{tos}): $$\Vert \phi_{\ep,\delta}\Vert_{C^{2,\al}(K)}\leq C\Big(\Vert \varphi_{\ep,\delta}\Vert_{L^\infty(K')} + \Vert h + \delta(u_{\ep,\delta} - u_0)\Vert_{C^{0,\al}(K')}\Big).$$ The right-hand side is uniformly bounded. Hence, $\varphi_\ep$ is a smooth solution to $$F_\ep(H(\varphi_\ep)) - c_{\ep} = h$$ for some $c_\ep \in \RR$.\\
\indent We wish to prove that $c_{\ep}=O(\ep^2)$. To this end, we need a better estimate on $\delta (u_{\ep,\delta} - u_0)$:
\begin{lemma}\label{lem:betterestimate}
There exists a constant $C$ that is independent of $\delta,\ep$ such that $$\delta \Vert u_{\ep,\delta} - u_0 \Vert_{L^{\infty}}\leq C\sqrt{\delta}+C\ep^2.$$
\end{lemma}
\begin{proof}
It suffices to prove that 
$$\delta \Vert u_{\ep,\delta,j} - u_0 \Vert_{L^{\infty}}\leq C\sqrt{\delta}+C\ep^2$$
on $\bar{\Omega}_j$ for $C$ independent of $j,\delta,\ep$. Note that by choosing $B= I$, $$ F_{\ep}(A)\leq \hat{F}_{\ep}(A)\leq \frac{1}{n}\Tr(A). $$
Thus the better barrier $\zeta_{n\delta}$ produced by Lemma \ref{lem:betterbarrier} satisfies
$$ F_{\ep}(H(\zeta_{n\delta}))-\delta(\zeta_{n\delta}-u_0)\leq 0<F_{\ep}(H(u_{\ep,\delta,j}))-\delta(u_{\ep,\delta,j}-u_0).$$
If the maximum of $v=u_{\ep,\delta,j}-\zeta_{n\delta}$ occurs at an interior point, then $v\leq 0$ throughout. If it is attained on the boundary of $\Omega_j$, then $$v\leq \Vert \zeta_{n\delta}- u_0\Vert_{L^{\infty}}\leq \frac{C}{\sqrt{\delta}} .$$ Thus $$u_{\ep,\delta}-\zeta_{n\delta}\leq\frac{C}{\sqrt{\delta}}.$$
Now we note that $$F_{\ep}(H(u_0))\geq m-C\ep^2\geq h-C\ep^2.$$
Thus, setting $\tilde{u}_0=u_0-\frac{C\ep^2}{\delta}$, and comparing with $u_{\ep,\delta,j}$ as above, we get the desired lower bound.
\end{proof}
Using Lemma \ref{lem:betterestimate} we see that $c_{\ep}\rightarrow 0$ as $\ep\rightarrow 0$. 
Moreover we have $\varphi_\ep(o) = 0$ and $\vert\nabla \varphi_\ep\vert \leq 1$. Then once again applying Arzela-Ascoli, there is a sequence $\ep_k\rightarrow 0$ such that $\varphi_k:= \varphi_{\ep_k}$ converge uniformly on compact sets to a Lipschitz continuous function $\phi$ with $\Lip(\phi) \leq 1$. 

\begin{lemma}
$\phi$ is psh and solves $$(dd^c\phi)^n = e^{-f}\omega^n$$ in the Bedford-Taylor sense. 
\end{lemma}
Combined with the integral estimate on $e^{-f}$, this proves Theorem \ref{thm:MA}. The proof of this lemma is inspired from \cite{Zer13}.
\begin{proof}
In the proof, we suppress the subscript $k$ on $\ep_k$. So for instance when we say let $\ep\rightarrow 0$ we actually mean that $\ep_k\rightarrow 0$. Since the claim is local, we fix a coordinate unit ball $B$. Then $\varphi_k\rightarrow \phi$ uniformly on $\overline{B}$. Choose smooth boundary data $\phi_m$ on $\partial B$ which converge uniformly to $\phi$ on $\partial B$. Let $v_m$ be a soluion to the Dirichlet problem $$\begin{cases}(dd^cv_m)^n = e^{-f}\omega^n\quad\text{ on $B$}\\ v_m = \phi_m\quad\text{ on $\partial B$}.  \end{cases}$$ We claim that for each fixed $m$,
\begin{equation}\label{eq:final-est}
\Vert \phi - v_m\Vert_{L^\infty(B)} \leq \Vert \phi - \phi_m\Vert_{L^\infty(\partial B)}.
\end{equation} Assuming this for now we immediately get that $v_m\rightarrow \phi$ uniformly on $B$. In particular $\phi$ is psh and also solves $$(dd^c\phi)^n = e^{-f}\omega^n.$$ In order to prove the claim, we first note that since $v_m$ is strictly psh and smooth on $\overline{B}$, $$K_m:= \{H(v_m)(x)~|~x\in \overline{B}\}$$ is compact in $\Herm^+(n,\CC).$ For a fixed $m$, there exists $\ep_m>0$ such that for all $\ep<\ep_m$ we have $$\sup_{A\in K_m}|F_\ep(A) - F(A)| \leq C\ep^2,$$ where recall that $C$ is independent of $K_m$ and hence $m$. Now, $F(H(v_m)) = h$ and so $$|F_\ep(H(v_m)) - h| \leq r_\ep,\quad r_{\ep}:= C\ep^2.$$ Recall that $\varphi_\ep$ solves $F(H(\varphi)) = h + c_\ep.$ Consider the ``test function" $q(z) = |z|^2 - 1$. This satisfies $$-1\leq q\leq 0 \text{ on $ B$},~q\equiv 0 \text{ on $\partial B$ and }H(q) \geq  \mu_BI,$$ where $\mu_B>0$ is some fixed constant depending only on $\omega\Big|_{\overline{B}}$. We set $$a_\ep:= \frac{r_\ep + |c_\ep|}{\mu_B} \text{ and }b_{m,\ep}:= \Vert \varphi_\ep - \phi_m\Vert_{L^\infty(\partial B)},$$ and consider the lower and upper barriers defined by $$\uv_{\ep,m} := v_m + a_\ep q - b_{m,\ep}\quad\uv_{\ep,m} := v_m - a_\ep q + b_{m,\ep}, \text{ where }.$$ Since $H(q) \geq 0$. it follows from \eqref{eq:finite-diff-F-ep} that $$F_\ep(H(\uv_{m,\ep})) \geq  F_\ep(\varphi_\ep)\geq F_\ep(H(\ov_{m,\ep})).$$ Indeed, by \eqref{eq:finite-diff-F-ep} 
\begin{align*}
F_\ep(H(\uv_{m,\ep})) &\geq F_\ep(H(v_m)) + a_\ep\mu_B \\
&\geq h - r_\ep + a_\ep\mu_B\\
&\geq  h + c_\ep = F_\ep(\varphi_\ep).
\end{align*}
The other estimate is proven similarly. On the boundary $\partial B$ we have $$\uv_{m,\ep} \leq \phi_\ep \leq \ov_{m,\ep}.$$  Then by the standard comparison principle (see for instance the argument in the proof of Lemma \ref{lem:comparison}) we have that $$\uv_{m,\ep} \leq \phi_\ep \leq \ov_{m,\ep}$$  on all of $B$. Since $|q| \leq 1$ this implies that $$\Vert\varphi_\ep - v_m\Vert_{L^\infty(B)} \leq a_\ep + b_{m,\ep}.$$ Fixing $m$ and letting $\ep \rightarrow 0$ proves the claim \eqref{eq:final-est}, since $a_\ep\rightarrow 0$.
\end{proof}

\end{document}